\documentclass[a4paper]{article}

\usepackage{amssymb,amsmath}
\usepackage{amsthm,multirow}
\usepackage{graphicx}
\usepackage{subfigure}
\usepackage[numbers]{natbib}
\usepackage{color}

\newcommand{\Pp}{\mathbb{P}}
\newcommand{\Ii}{\mathbb{I}}
\newcommand{\Ee}{\mathbb{E}}
\newcommand{\Ww}{\mathcal{W}}

\newcommand{\Ff}{\mathcal{F}}
\newcommand{\Oo}{\mathcal{O}}
\newcommand{\Xx}{\mathcal{X}}
\newcommand{\Aa}{\mathcal{A}}

\newtheorem{guess}{Conjecture}[section]
\newtheorem{definition}[guess]{Definition}
\newtheorem{theory}[guess]{Theorem}
\newtheorem{lemma}[guess]{Lemma}
\newtheorem{corol}[guess]{Corollary}

\begin{document}
\title{Voting Power : A Generalised Framework}
\author{Sreejith Das and Iead Rezek$^{*}$\\$^{*}$Department of Engineering Science, University of Oxford, U.K. }
\maketitle 

\begin{abstract}
This paper examines an area of Game Theory called Voting Power Theory.  With the adoption of a measure theoretic framework it argues that the many different indices and tools currently used for measuring voting power can be replaced by just three simple probabilities.  The framework is sufficiently general to be applicable to every conceivable type of voting game, and every possible decision rule. 
\end{abstract}


\section{Introduction}
\label{sect:introduction}

We are all familiar with the idea of voting.  It affects every part of our lives, from village committees deciding trivial matters, to the boardroom opting for redundancies, and even government cabinets choosing a path to war.  Clearly, everyone, everywhere, is subject to the decisions made, or not made, by voting games.  Arguably, they are the most important, and influential, type of game studied by game theorists.  Despite this, there is one aspect of voting games which is poorly understood.  Namely, how to go about constructing a democratically fair voting game.

The recent events dubbed the ``Arab Spring'' highlight how strongly ordinary people can feel about democracy.  Even in those societies, for which, some have claimed, democracy is an alien concept, we find those that are prepared to lay down their lives in its pursuit.  Despite their fervent ardour for democracy, it is not entirely clear what a good democracy is, or how one would go about creating such a thing. 

However, the one thing we can all agree upon is that a democracy that doesn't treat all of its citizens equally is no democracy at all.  As such, many would argue that fairness is one of the founding principles of any modern democracy, the idea of \emph{equal representation for all}.  In other words, everyone must have the same equal ability to influence the outcome of a political decision.  We will term this ability the `Voting Power' of a voter.  

How do you go about measuring this voting power?  The importance of this question has motivated a number of researchers to devise quantitative measures of fairness in election systems. The first paper on voting power was written from a statistical perspective by~\cite{Penrose1946}, unfortunately this paper was largely ignored.  Later on,~\cite{ShapleyShubik1954} proposed a game theoretic method for measuring the a priori power of a voter.  In contrast to both of these approaches, we have decided to take a measure theoretic approach to voting power.  Our reasoning is simple.  We aim to unify the many different strands of voting power research, with one, all encompassing, methodology.  The measure theoretic analysis is applicable to every type of voting game, with every type of voting rule, from simple ``yes/no'' majority voting, to multi-candidate games with voter abstentions.  The key result we will produce is the exciting revelation that the many different techniques can be replaced by just three fundamental probabilities.  

The paper is structured as follows. We first propose an intuitive and statistical interpretation of voter influence, and show that it can be expressed using just three elementary probabilities. We then flesh out the concept of influence by defining events in which the voter is critical to the outcome of the election. We then show that these events, which we call critical events, are at the heart of all voting power techniques.  Having expressed voting power as functions of critical events, we then use measure theory to compute the probabilities of these events, and thus are able to provide, for the first time, a probabilistic representation of the existing voting power techniques. This allows us to see exactly what these techniques measure, and how the various techniques relate to one-another. The paper concludes with a brief discussion of the implications of our work.

\section{A Brief History of Voting Power}
\label{sect:A Brief History of Voting Power}

In a national referendum, giving everyone the same voting power is achieved by following the principle of `one person, one vote'.  If all decisions were made by national referendum, this would suffice to ensure fairness.  However, in many countries we elect and appoint representatives to collectively make decisions on our behalf.  We do not vote upon every new law or budget proposal.  For constituencies of unequal size, we are faced with the problem of finding a voting weight that adheres to our guiding principles of fairness.  Perhaps we should make the voting weight of a representative directly proportional to the population they represent?


Unfortunately, the voting power of a representative is not proportional to their voting weight, as was demonstrated in the seminal work by~\cite{Banzhaf1965}, and reproduced in the following table. 

\begin{center}
\begin{tabular}{*{4}{c}}
Representative & Population & No. of Weighted Votes & Relative Voting Power \\\hline
A & 40,000 & 4 & 7 \\
B & 20,000 & 2 & 1\\
C & 10,000 & 1 & 1 \\
D & 10,000 & 1 & 1 \\
E & 10,000 & 1 & 1 \\
\end{tabular}
\end{center}

One can see that representative A, along with any other single representative can combine to make a majority.  Furthermore, there is only one winning coalition that doesn't include representative A (coalition BCDE).  Thus, using only intuitive argument, and without recourse to voting power theory, it becomes apparent that representative A has considerably more voting power than everyone else, and that the others must share the same minimal voting power. 

With their publication,~\cite{ShapleyShubik1954} presented a method for calculating voting power, loosely based upon assigning a value to a winning coalition, and then distributing that value among those within it.   Other interpretations of voting power have also been proposed, namely~\cite{Banzhaf1965},~\cite{Coleman1971},~\cite{DeeganPackel1978}, and~\cite{Johnston1978}. ~\cite{Straffin1977, Straffin1978, Straffin1979} realised that many of these techniques were, in fact, measuring the same thing, albeit with different underlying probability models.  He even proposed that the Banzhaf and Shapley-Shubik techniques were equivalent. (However, we will later show that the Banzhaf measure and the Shapley-Shubik index are inequivalent, one being a measure of Total Criticality, and the other being a measure of Increasing Criticality - the different types of criticality are discussed in Section \ref{subsect:criticality}).

Despite having a plethora of different techniques to calculate voting power there remain a number of challenges. First, the different techniques can give wildly different results for the same game, making it difficult to know which one to trust~\citep{Straffin1977}.  Second, these techniques are often restricted to games with binary voting choices (``yes'' or ``no''), and do not account for abstentions. And third, all the techniques make an implicit assumption about the probability distribution of the voters, making them impossible to use in a game with a different probability model.  As a consequence, there does not appear to be a universally accepted method of measuring fairness in election systems.  

The starting point for many of the proposed indices was a set of ``intuitive notions'' of what constitutes voting power.  Unsurprisingly, basing a subject upon intuitive notions, instead of axiomatic principles, has led to a debate over which technique is best~\citep{Banzhaf1965,Johnston1978,Laver1978a,FelsMach1998,GelmanKatzTuerlinckx2002,Leech2002b,Paterson2004,Lindner2008}.

The current debate about voting power techniques is reminiscent of the early history of Artificial Neural Networks. Developed using biological considerations, there was much debate over what they could classify, and how they could be parameterised.  Only when they were placed on a common mathematical footing, and shown to be function approximators \citep{Cybenko1989, Hornik1991}, did the debate end.  Similarly, in our considered opinion, existing ways of quantifying voting power are black box techniques, derived on the basis of a restricted and subjective set of dictums. They were not derived from first principles, and so little is known about their validity. This leaves all analysis using these techniques open to debate.

Just as the artificial neural network debate was silenced by showing that they are all function approximators, we hope to end the debate within voting power by showing that all the existing techniques are simple probabilities. And to that end, we approach the subject afresh from basic principles.

\section{A Intuitive Approach to Voting Power}
\label{sect:A Intuitive Approach to Voting Power}

Ask a doctor if you should give up smoking, and they will tell you that smoking increases your chance of dying from lung cancer by, say, 40\%.  This example underlines the fact that probabilistic statements are commonplace within society.  And, more to the point, we are all familiar with the notion of measuring influence using the probability of an outcome conditioned on a controlling factor.  (In this case, the outcome is death, and the probability is conditioned on smoking).  

As we are concerned with voting games, we are interested in two separate notions of influence: how much can a player (voter) increase the likelihood of the outcome, and how much can a player decrease the likelihood? We can define our intuitive notions of influence as follows.

The ability of a player $i$ to positively influence the outcome is,
\begin{equation}
\Pr(\mathrm{Outcome}   \, | \, \mbox{player $i$ does all to ensure the outcome}) - \Pr(\mathrm{Outcome}).
\label{eqn:posgameoutcome}
\end{equation}
And the ability of a player $i$ to negatively influence the outcome is,
\begin{equation}
\Pr(\mathrm{Outcome}) - \Pr(\mathrm{Outcome} \, | \, \mbox{player $i$ does all to prevent the outcome} ).
\label{eqn:neggameoutcome}
\end{equation}

On the basis of both negative and positive measures of influence, we argue that the total amount of influence a player $i$ has, with respect to a specific outcome $O$, is given by the sum of these two expressions, i.e. 
\begin{equation}
\Pr(O \, | \, \mathrm{player} \; i \;  \mathrm{does \; all \; to \; ensure \;}O) - \Pr(O \, | \, \mathrm{player} \; i \;  \mathrm{does \; all \; to \; prevent \;} O).
\label{eqn:totalgameoutcome}
\end{equation}

The above probability statements are very intuitive and basic notions of influence.  They can be easily understood by the general population, as probabilities are fairly standard in every-day life.  Thus, a probabilistic measure of voting power is both desirable, and arguably, essential.

We strongly believe that voting power analysis should be carried out using these probabilities.  As the increased transparency of probabilities compares favourably to the opaqueness of the currently used standard techniques.  This can only help the subject matter be more widely accepted, and ultimately lead to better democracies.  In order to encourage this we will show, in Section \ref{sect:vptr}, that all of the standard techniques are calculating expressions~(\ref{eqn:posgameoutcome}),~(\ref{eqn:neggameoutcome}) and~(\ref{eqn:totalgameoutcome}).

\section{Criticality}
\label{sect:criticality}

The influence of a player on a game's outcome can, in some sense, be measured by the ability of the player to change the outcome.  In which case, we talk about a player being \emph{critical} to the outcome.  

Consider a game with eleven players and a simple majority decision rule.  On a particular issue, player $i$ is in favour, and the motion is passed $8-3$.  If player $i$ were to change its mind, and vote against, the motion will still pass $7-4$.  Clearly $i$ is not \emph{critical} in this scenario, it has no influence on the eventual outcome.  Now consider a different situation in which player $i$ votes in favour, and the motion is passed $6-5$.  If player $i$ were to change its mind, and vote against, the motion will be rejected $5-6$.  Clearly $i$ is \emph{critical} in this scenario.  It has considerable influence on the eventual outcome.
 
In a game in which a player can only vote ``yes'' or ``no'', a player $i$ can be critical in two distinct ways.  It can be critical by increasing its support, i.e. it joins a losing coalition to make it winning. Alternatively, it can can be critical by decreasing its support, i.e. it leaves a winning coalition to make it losing. We refer to these scenarios as increasing criticality and decreasing criticality, respectively. The total influence a player can have on the outcome is the combination of both criticalities, which we term total criticality.

There is one further concept which we briefly describe here, and rigorously define later in section~\ref{subsect:criticality}, it is the idea of criticality assumptions.  These criticality assumptions are needed in games where players have more than two choices, and allow us to extend our analysis to more general settings.  If we assume that a player is either initially voting ``no'', or finally voting ``no'', then we term this Criticality $0$.  Alternatively, if we make no assumption as to how a player will initially, or finally, vote, then we term this Criticality $\delta$.  The combination of the different criticalities and criticality assumptions leads to the notions of criticality as listed in the following table.

\begin{center}
\begin{tabular}{|c|c|c|c|}\hline
& Increasing & Decreasing & Total \\\hline
\multirow{2}{*}{Criticality 0} & Increasing  &  Decreasing   &  Total  \\
 &  Criticality 0 ($IC^{0}$) &   Criticality 0 ($DC^{0}$)  &   Criticality 0 ($TC^{0}$)  \\\hline
\multirow{2}{*}{Criticality $\delta$} &  Increasing  &  Decreasing  & Total \\
 &  Criticality $\delta$  ($IC^{\delta}$) &  Criticality $\delta$ ($DC^{\delta}$)  &  Criticality $\delta$ ($TC^{\delta}$) \\\hline
\end{tabular}
\end{center}

A more in-depth, and intuitive discussion of criticality can be found in~\cite{Das2011a}.


\section{Criticality and Standard Voting Power Techniques}
\label{sect:vpt}

The concept of criticality, as we will now demonstrate, is at the very heart of all voting power techniques.  In fact, the standard techniques used for measuring voting power can be viewed as simply counting the number of voting scenarios in which a player is considered critical.  

Without doubt, the two most widely used voting power techniques are the Banzhaf measure, and the Shapley-Shubik index.  In this section we will examine these techniques in greater detail, along with a selection of other less widely used techniques.  All these techniques were originally proposed for a small subset of voting games in which every player must vote ``yes'' or ``no'', there is no concept of abstention.  Likewise, there are only two possible outcomes to the game, either winning, or losing.  

\subsection{Shapley-Shubik Technique}

~\cite{ShapleyShubik1954} state that the power of an individual member of a legislative body depends on the chance they have of being critical to the success of a winning coalition.  They explain that a voter can be ``pivotal'' when they can turn a possible defeat into a success.  And they construct their index as follows:

\begin{enumerate}
\item There are a group of individuals all willing to vote for some bill.
\item They vote in order.
\item As soon as a majority has voted for it, it is declared passed.
\item The (pivotal) member who voted last is given credit for passing the bill.
\end{enumerate} 

The voting orders are chosen randomly, and they calculate the number of times that a voter is considered pivotal.  The final Shapley-Shubik index is produced by dividing the pivotal count by the total number of voting orders (i.e. $n!$, where $n$ is the number of voters).  They describe this as the frequency with which a particular voter is considered pivotal.  

For a moment, let's examine their term pivotal.  It requires a losing voting scenario in which the voter expresses zero support towards the bill to become winning when they increase their support.  Rather than call the voter pivotal, let's call it critical instead.  Furthermore, as the voter becomes critical by increasing its support, let's call it increasingly critical.  Finally, as the pivotal voter always starts off by expressing zero support for the bill, it should be called increasing criticality zero.  Hence, the Shapley-Shubik index is given by the following algorithm.

\begin{enumerate}
\item Examine every possible voting order.
\item For each voting order identify if it is increasing criticality zero for the given voter.
\item If so, add $1$ to a running count for the given voter.
\item Repeat until all voting orders have been examined, then divide by $n!$.
\end{enumerate} 

It is explicit within the construction of the Shapley-Shubik index that all voting orders are equiprobable, the term $\frac{1}{n!}$ is the probability of a given voting order arising.  With this is mind, it is easy to see that the Shapley-Shubik index is nothing more than the probability of a voter being increasing criticality zero.  If we use the symbol $\omega$ to represent a voting order, then,

\[
\mathrm{ShapleyShubik} = \int_{\omega} \; \; \Ii^{IC^{0}}(\omega) \; \; \Pr(d\omega) \; = \; \Pr(IC^{0}),
\]
where
\[ 
\Ii^{IC^{0}}(\omega) = \left\{ \begin{array}{ll}
1 & \mbox{if $\omega$ is increasing criticality zero for the given voter;} \\
0 & \mbox{otherwise}. \end{array}
\right. 
\]

\subsection{Banzhaf Technique}

~\cite{Banzhaf1965} states that power in a legislative sense is the ability to affect outcomes.  He says specifically the power of a legislator is given by the number of possible voting combinations of the entire legislature in which the legislator can alter the outcome by changing their vote.

We can interpret the ability to alter the outcome through a change of vote as follows: a voter is able to make a losing outcome winning by increasing their support (increasing criticality), or a voter is able to make a winning outcome losing by decreasing their support (decreasing criticality). The combination of increasing and decreasing criticality is called total criticality.  Furthermore, as Banzhaf makes no specific requirement for the voter to be initially voting one way or the other, let's call this total criticality delta.  Hence the Banzhaf measure of power is given by the following algorithm.

\begin{enumerate}
\item Examine every possible voting combination.
\item For each voting combination identify if it is total criticality delta for the given voter.
\item If so, add $1$ to a running count for the given voter.
\item Repeat until all voting combinations have been examined, then divide by ${2^{n}}$.
\end{enumerate} 

Banzhaf assumes that every voting combination is equiprobable, the term $\frac{1}{2^{n}}$ is the probability of a given voting combination arising (where $n$ is the number of players).  With this is mind, it is easy to see that the Banzhaf measure is nothing more than the probability of a voter being total criticality delta.  If we use the symbol $\omega$ to represent a voting combination, then,

\[
\mathrm{Banzhaf} = \int_{\omega} \; \; \Ii^{TC^{\delta}}(\omega) \; \; \Pr(d\omega) \; = \; \Pr(TC^{\delta}).
\]

Where, 

\[ 
\Ii^{TC^{\delta}}(\omega) = \left\{ \begin{array}{ll}
						1 & \mbox{if $\omega$ is total criticality delta for the given voter;} \\
						0 & \mbox{otherwise}.
					\end{array}
			\right. 
\]

\subsection{Straffin}

~\cite{Straffin1977} proposed two different techniques differentiated by the probability model assumed.  The Independence Assumption technique uses a uniform probability distribution, while the Homogeneity Assumption technique assumes all the players vote in favour with the same probability $p \in [0,1]$.  

Straffin defines his measure as the probability that player $i$'s vote will make a difference in the outcome.  Making it, like Banzhaf, a measure of total criticality.  And, as there is no requirement for player $i$ to be initially voting one way or another, it is a measure of total criticality delta.

\[
\mathrm{Straffin} = \int_{\omega} \; \; \Ii^{TC^{\delta}}(\omega) \; \; \Pr(d\omega) \; = \; \Pr(TC^{\delta}).
\]

Both the Independence Assumption technique, and the Homogeneity Assumption technique are given by $\Pr(TC^{\delta})$.  The different probability models are absorbed by the $\Pr(d\omega)$ term.

\subsection{Coleman}

Of all the researchers working in the field of voting power theory, \cite{Coleman1971} was perhaps the first to appreciate the subtle differences that exist between increasing and decreasing criticality~(while \cite{ShapleyShubik1954} understood it was possible to be decreasingly critical, they did not appreciate that this was materially different to being increasingly critical).  He defined two measures of power, the power to initiate action, and the power to prevent action.  

The initiate action measure is a count of the number of times a player can be critical given the coalition is losing.  Hence, it is a measure of increasing criticality conditioned on a coalition being losing.  If we let $\overline{WIN}$ represent the set of losing coalitions, and $\lambda_{\overline{WIN}}$ be the sigma finite marginal measure on $\overline{WIN}$, then:

\[
\mathrm{Coleman \; Initiate \; Action} = \int_{\omega \in \overline{WIN} }  \; \; \Ii^{IC^{\delta}}(\omega) \; \; \lambda_{\overline{WIN}} \; (d\omega) \; = \; \Pr(IC^{\delta} | \; \overline{WIN} ).
\]

The prevent action measure is a count of the number of times a player can be critical given the coalition is winning.  Hence, it is a measure of decreasing criticality conditioned on a coalition being winning.  If we let $WIN$ represent the set of winning coalitions, and $\lambda_{WIN}$ be the sigma finite marginal measure on $WIN$, then:

\[
\mathrm{Coleman \; Prevent \; Action} = \int_{\omega \in WIN } \; \; \Ii^{DC^{\delta}}(\omega) \; \; \lambda_{WIN} \; (d\omega) \; = \; \Pr(DC^{\delta} | \; WIN  ).
\]

%
%
%
%
%
%
%
%
%

\subsection{Johnston}

The~\cite{Johnston1978} index can be described as follows.  Examine every winning coalition, identify those members which can destroy the coalition and allocate a point, or fraction of a point, to them.  In other words, this is a measure of decreasing criticality.  From his paper, it seems reasonable to assume that his index requires the player to express zero approval in order to destroy the coalition, so we will call it a decreasing criticality zero measure.

\[
\mathrm{Johnston} = \int_{\omega} \; \; \Ii^{DC^{0}}(\omega) \; \; \Pr(d\omega) \; = \; \Pr(DC^{0}).
\]

Both the original version of the Johnston index (where one point is added for every destroyable coalition), and the modified version (where a fraction of a point is added) are given by $\Pr(DC^{0})$.  In the modified version, the fraction that is added is a function of $\omega^{N}$ only, hence it can be absorbed within the $\Pr(d\omega)$ term.  Ergo, the modified version is the same as the original version, albeit with a different probability model. The actual fraction that is added is inversely proportional to the number of players that express full support in $\omega^{N}$.  Hence, the probability model of the modified index implies that coalitions with more players expressing full support are less likely to occur.

\subsection{Summary}

This section examined several of the most popular voting power techniques (a discussion of the Deegan-Packel and the Holler Public Good index is left until Appendix \ref{appendix:dphp}).  A fundamental flaw inherent in all these techniques is their dependency upon a specific probability model.  Which makes it almost impossible to use these techniques in a game with a different probability distribution.  However, this failing is easily overcome by using our measure theoretic interpretation instead, as it is defined independent of the underlying probability distribution. 

Any voting power technique that can be calculated by an algorithm which iterates through a set of voting scenarios, testing each one in turn to see if they are critical, can be expressed within our framework.  In a sense, the standard techniques are just specific instances of our general measure theoretic approach.  This is a fundamental point, because it means that any such voting power technique is subject to the analysis and results given in this paper.  We are not aware of any voting power technique for which our results do not apply.

\section{Basic Definitions}
\label{sec:Basic Defs}
Rather than restrict our analysis to a specific voting system, we will introduce here the concept of a generalised voting game.  This generalised voting game encompasses all possible voting games of interest, in that it allows for any voting rule, any number of possible voting outcomes, and any probability distribution of the players.  The definitions of probability and product spaces are taken from~\cite{Pollard2003}.

\subsection{Voting Games}

\begin{definition}
A \textbf{player} is a probability space $(\Xx_{i},\Aa_{i},\Pp_{i})$, where $\Xx_{i}$ is a set, $\Aa_{i}$ is a sigma-field of subsets of $\Xx_{i}$, and $\Pp_{i}$ is a countably additive, nonnegative measure with $\Pp_{i}(\Xx_{i})=1$.  Given a set of $N$ players, where $|N|=n$, the set of all ordered $n$-tuples $(x_{1},\ldots,x_{n})$, with $x_{j} \in \Xx_{j}$ for each $j \in 1, \ldots n$ is denoted as $\Xx_{1} \times \cdots \times \Xx_{n}$ and abbreviated to $\Omega^{N}$.  Given a player $i$, the set of all ordered $(n-1)$-tuples $(x_{1},\ldots,x_{i-1},x_{i+1},x_{n})$, with $x_{j} \in \Xx_{j}$ for each $j \in 1, \ldots,i-1,i+1,\ldots n$ is denoted as $\Xx_{1} \times \cdots \times \Xx_{i-1} \times \Xx_{i+1} \times \cdots \times \Xx_{n}$ and abbreviated to $\Omega^{N \setminus \{i\}}$.  The action of creating a single $(n-1)$-tuple, denoted as $\omega^{N \setminus \{i\}}$, from a single $n$-tuple $\omega^{N}$ by removing the element $x_{i}$ is represented as $\omega^{N} \setminus x_{i}$.  The action of creating a single $n$-tuple, denoted as $\omega^{N}$, from a single $(n-1)$-tuple $\omega^{N \setminus \{i\}}$ by adding an element $x_{i} \in \Xx_{i}$ is represented as $\omega^{N \setminus \{i\}} \times x_{i}$.
\end{definition}

(When there is no risk of confusion the superscript will be dropped from the set $\Omega$). 

The key concepts from the previous definition are: a player $i$ can vote by expressing one of $\{x_{i}\} \in \Xx_{i}$, $\omega^{N}$ represents a voting configuration (an event) with $|N|$ players, $\omega^{N} \setminus \{x_{i}\}$ represents a voting configuration with player $i$ removed, and $\omega^{N \setminus \{i\}} \times \{x_{i}\}$ represents a voting configuration in which player $i$ has joined by expressing $\{x_{i}\}$.


\begin{definition}
Given a set of $N$ players, where $|N|=n$, a set of the form $A_{1} \times \cdots \times A_{n} = \{(x_{1},\ldots,x_{n}) \in \Xx_{1} \times \cdots \times \Xx_{n} : x_{i} \in A_{i}$ for each $i\}$, with $A_{i} \in \Aa_{i}$ for each $i$, is called a measurable rectangle.  The product sigma field $\Aa_{1} \times \cdots \times \Aa_{n}$ on $\Xx_{1} \times \cdots \times \Xx_{n}$ is defined to be the sigma field generated by all measurable rectangles.  Let the product space $(\Xx_{1} \times \cdots \times \Xx_{n}, \Aa_{1} \times \cdots \times \Aa_{n})$ be denoted as $(\Omega,\mathcal{F})$.  
\end{definition}

\begin{definition}
A \textbf{generalised voting game} is a quadruple $(\Omega,\Ff,\Pp, \Ww)$ such that $(\Omega,\Ff,\Pp)$ is the product space generated by a set of $N$ players, $\Pp$ is the product measure, and $\Ww$ is a $\Ff \setminus \Oo$ measurable function, where the elements $O \in \Oo$ are called \textbf{outcomes}.  Such a game is denoted as a $\mathbf{GVG}(\Omega,\Ff,\Pp, \Ww)$.
\label{def:gvg}
\end{definition}

Before moving forward it might be useful to examine the definition of a $GVG$ in greater detail. First, lets take a closer look at the players.  Each one is defined as a probability space, beyond this there is no further restriction.  As such, it can model every possible way of voting, from simple ``yes/no'' with abstentions, to a selection from a continuous set.  Furthermore, there is no requirement for any kind of ordering to exist on the set of possible player actions.  

Now let's examine the voting rule, which is defined by the measurable function $\Ww$.  Beyond the requirement of measurability there is no further restriction.  Hence, it encompasses every possible mapping from the set of possible voting configurations to the set of possible outcomes.  This includes, simple pass/reject outcomes, to a collective full ranking of preference across multi-candidate outcomes.  

We believe that this generalised definition of a voting game encompasses every conceivable real life voting game that one could possibly wish to analyse.

\subsection{Criticality}
\label{subsect:criticality}

Along with the generalised definition of a voting game we have just introduced, we require a generalised definition of criticality.

\subsubsection{Criticality Sets}

\begin{definition}
For a $GVG(\Omega,\Ff,\Pp, \Ww)$, a player $i$ is \textbf{increasingly critical} with respect to an outcome $O \in \Oo$ in an event $\omega^{N} \in \Omega^{N}$ if, and only if, $\Ww(\omega^{N}) \neq O$ and there exists an $\{x'_{i}\} \in \Xx_{i}$ such that $\Ww((\omega^{N} \setminus \{x_{i}\}) \times \{x'_{i}\}) = O$.  Let $O \_ IC_{i}$ denote the set of increasingly critical events for a player $i$ with respect to an outcome $O$.
\label{def:incc}
\end{definition}

\begin{definition}
For a $GVG(\Omega,\Ff,\Pp, \Ww)$, a player $i$ is \textbf{decreasingly critical} with respect to an outcome $O \in \Oo$ in an event $\omega^{N} \in \Omega^{N}$ if, and only if, $\Ww(\omega^{N}) = O$ and there exists an $\{x'_{i}\} \in \Xx_{i}$ such that $\Ww((\omega^{N} \setminus \{x_{i}\}) \times \{x'_{i}\}) \neq O$.  Let $O \_ DC_{i}$ denote the set of decreasingly critical events for a player $i$ with respect to an outcome $O$.
\label{def:decc}
\end{definition}

\begin{definition}
For a $GVG(\Omega,\Ff,\Pp, \Ww)$, a player $i$ is \textbf{totally critical} with respect to an outcome $O \in \Oo$ in an event $\omega^{N} \in \Omega^{N}$ if it is either increasingly critical or decreasingly critical, with respect to the aforementioned outcome and event.  Let $O \_ TC_{i}$ denote the set of totally critical events for a player $i$ with respect to an outcome $O$.  For any given event $\omega^{N}$, it is not possible to be simultaneously both increasingly and decreasingly critical with respect to a given outcome $O$, therefore $(O \_ IC_{i} \cap O \_ DC_{i}) = \emptyset$.
\label{def:totc}
\end{definition}

Increasing Criticality measures a player's ability to create the outcome they want, while Decreasing Criticality measures their ability to prevent an outcome they don't want.  Total Criticality, as the combination of Increasing and Decreasing Criticality, is a measure of a player's total ability to influence an outcome of the game.

In simple ``yes/no'' voting games, there is a bijection between the Increasing Criticality and the Decreasing Criticality events.  However, if any of the players are allowed to abstain, this symmetry can be broken, making it vital to understand which criticality is being measured.

\subsection{Criticality Assumptions}

Definitions \ref{def:incc}, \ref{def:decc}, and \ref{def:totc} measure criticality with respect to two events for a given player $i$.  Criticality assumptions place restrictions on player $i$, governing how it can change its vote between these two events.  

\begin{definition}
\textbf{Criticality} $\mathbf{\delta}$ - With this assumption there is no restriction on how player $i$ can vote between the two different events that define it as critical.  The set of criticality $\delta$ increasingly critical events for player $i$, with respect to an outcome $O$, is denoted by $O \_ IC_{i}^\delta$, and the set of criticality $\delta$ decreasingly critical events for player $i$, with respect to an outcome $O$, is denoted by $O \_ DC_{i}^{\delta}$.
\label{def:critD}
\end{definition}

Criticality $\delta$ is suitable for every type of game, whereas Criticality $0$ (see below) is more appropriate for games in which there is an element of cost, or risk, involved in supporting a decision.  

\begin{definition}
\textbf{Criticality 0} - With this assumption one of the two events that define player $i$ as being critical must have player $i$ voting with its lowest possible support for outcome $O$.  The set of criticality 0 increasingly critical events for player $i$, with respect to an outcome $O$, is denoted by $O \_ IC_{i}^{0}$, and the set of criticality 0 decreasingly critical events for player $i$, with respect to an outcome $O$, is denoted by $O \_ DC_{i}^{0}$.
\label{def:crit0}
\end{definition}

In simple ``yes/no'' voting games Criticality $0$ and Criticality $\delta$ are equivalent.  However, if any of the players are allowed to abstain this equivalence will be lost, and it will be necessary to understand which criticality assumption is being measured.

\subsection{$x_{i}^{O_{\mathrm{max}}}$ and $x_{i}^{O_{\mathrm{min}}}$}

The final piece of notation required before the analysis can begin in earnest is a definition of the elements $x_{i}^{O_{\mathrm{max}}}$ and $x_{i}^{O_{\mathrm{min}}}$.  They are the generalised equivalents of voting ``yes'' and ``no'' in a traditional voting game.

\begin{definition}  For a $GVG(\Omega,\mathcal{F},\Pp, \Ww)$, a player $i$, and an outcome $O \in \mathcal{O}$, let $\Ii^{O} : \Omega^{N} \rightarrow \{\{0\},\{1\}\}$ be the indicator function that an event $\omega^{N}$ is classified as outcome $O$, i.e. when $\Ww(\omega^{N}) = O$.  Then, given an $\omega^{N \setminus \{i\}} \in \Omega^{N \setminus \{i\}}$, define $\{x_{i}^{O_{\mathrm{max}}}\}$ such that for all $x_{i} \in \Xx_{i}$,

\[ \Ii^{O} \left( \omega^{N \setminus \{i\}} \times \; \{x_{i}^{O_{\mathrm{max}}}\} \right) \geq \Ii^{O} \left( \omega^{N \setminus \{i\}} \times \; x_{i} \right). \]

Likewise, define $\{x_{i}^{O_{\mathrm{min}}}\}$ such that for all $x_{i} \in \Xx_{i}$,

\[ \Ii^{O} \left( \omega^{N \setminus \{i\}} \times \; \{x_{i}^{O_{\mathrm{min}}}\} \right) \leq \Ii^{O} \left( \omega^{N \setminus \{i\}} \times \; x_{i} \right). \]

\label{def:imaxmin}
\end{definition}

Clearly $x_{i}^{O_{\mathrm{min}}}$ and $x_{i}^{O_{\mathrm{max}}}$ need not be unique elements within $\Xx_{i}$, and could instead be subsets.  Should this turn out to be the case, the elements $\{x_{i}^{O_{\mathrm{min}}}\}$ and $\{x_{i}^{O_{\mathrm{max}}}\}$ can be taken as any appropriate element within said subsets.  

\section{A Measure Theoretic Analysis of Voting Power}
\label{sect:results}

In this section we will analyse the different notions of criticality as they are applied in the generalised voting game.  This will allow us to create an all encompassing framework for measuring voting power.  Given that the standard techniques we introduced in Section \ref{sect:vpt} are nothing more than examples of criticality analysis in specific instances of our generalised voting game, it naturally follows that the results we present here include, as specific instances, the standard techniques.

In Section \ref{sect:vpt} we expressed each of the techniques as an integral of a criticality based indicator function.  In order to carry out these integrations we will require the following lemma.

\begin{lemma}
For a $GVG(\Omega,\mathcal{F},\Pp, \Ww)$, a player $i$, and any integrable function~$f$,

\[
\int_{\omega^{N} \in \Omega^{N} } f(\omega^{N}) \; \Pp(d\omega^{N})  = \int_{\omega^{N \setminus \{i\}} \in \Omega^{N \setminus \{i\}} } \int_{x_{i} \in \Xx_{i}} f(\omega^{N}) \; \; \mu_{\omega^{N \setminus \{i\}}} (dx_{i}) \; \lambda(d\omega^{N \setminus \{i\}}). 
\]

where $\Pp$ is the sigma finite measure on the probability space $\Omega^{N}$, 
$\lambda$ is the sigma finite marginal measure on $\Omega^{N \setminus \{i\}}$, and $\mu_{\omega^{N \setminus \{i\}}}$ is the sigma finite marginal measure on $\Xx_{i}$, given $\omega^{N \setminus \{i\}}$.
\label{lem:multiint-i}
\end{lemma}

\begin{proof}
This follows from Definition \ref{def:gvg} with the realisation that $(\Omega,\mathcal{F},\Pp)$ is the product space made up of the individual players, and a subsequent application of Fubini's Theorem.  
\end{proof}

\subsection{Criticality $\delta$}

\begin{lemma}
For a $GVG(\Omega,\mathcal{F},\Pp, \Ww)$ and a player $i$, 

\[
\Pr(O \_ DC_{i}^{\delta})  =  \Pr(O) - \Pr(O | \{x_{i}^{O_{\mathrm{min}}}\}).
\]

\label{lem:prdcDprobv1}
\end{lemma}

\begin{proof}
Let $\omega^{N \setminus \{i\}} = \omega^{N} \setminus x_{i}$, then by Definitions \ref{def:decc}, \ref{def:critD} and \ref{def:imaxmin}, the indicator function $\Ii^{O \_ DC_{i}^{\delta}} : \Omega^{N}  \rightarrow \{\{0\},\{1\}\}$ for the set $O \_ DC_{i}^{\delta}$  is given by,

\[ \Ii^{O \_ DC_{i}^{\delta}}(\omega^{N}) = \left\{ \begin{array}{ll}
						1 & \mbox{if $\Ww(\omega^{N}) = O$ and $\Ww(\omega^{N \setminus \{i\}} \times \{x_{i}^{O_{\mathrm{min}}} \}) \neq O$}; \\
						0 & \mbox{otherwise}.
					\end{array}
			\right. \]

$\Ii^{O \_ DC_{i}^{\delta}}$ can be expressed as follows, where $\Ii^{O}$ is the indicator function for the event being classified as outcome $O$,

\begin{equation}
\Ii^{O \_ DC_{i}^{\delta}} = \Ii^{O}( \omega^{N}) - \Ii^{O}(\omega^{N \setminus \{i\}} \times \{x_{i}^{O_{\mathrm{min}}} \} ).
\label{eq:inddciD}
\end{equation}

Integrating $\Ii^{O \_ DC_{i}^{\delta}}$ over all the events in the $GVG$ creates the expectation of the random variable $\Ee \Ii^{O \_ DC_{i}^{\delta}}$, which we can interpret as the probability of an event $\omega^{N}$ being in the set $O \_ DC_{i}^{\delta}$.

\begin{eqnarray}
\Pr(O \_ DC_{i}^{\delta}) & = & \int_{\omega^{N} \in \Omega^{N}} \Ii^{O}(\omega^{N}) \; - \; \Ii^{O}(\omega^{N \setminus \{i\}} \times \{x_{i}^{O_{\mathrm{min}}}\}) \; \; \Pp (d\omega^{N}). \nonumber \\
\Pr(O \_ DC_{i}^{\delta}) & = & \int_{\omega^{N} \in \Omega^{N}} \Ii^{O}(\omega^{N}) \; \; \Pp (d\omega^{N}) \; \nonumber \\
&&  \hspace{.85cm} - \; \int_{\omega^{N} \in \Omega^{N}} \Ii^{O}(\omega^{N \setminus \{i\}} \times \{x_{i}^{O_{\mathrm{min}}}\}) \; \; \Pp (d\omega^{N}). \nonumber \\
\Pr(O \_ DC_{i}^{\delta}) & = & \mathrm{I}1 - \mathrm{I}2. \label{eq:dcDI2v1}
\end{eqnarray}

Here, $\mathrm{I}1$ is the expectation of the random variable $\Ee \Ii^{O}$, which we can interpret as the probability of an event $\omega^{N}$ being in the set $O$.  

\begin{equation}
\mathrm{I}1  =  \Pr(O). \label{eq:dcDI2bv1}
\end{equation}

Using Lemma \ref{lem:multiint-i}, $\mathrm{I}2$ can be expressed as,

\[
\mathrm{I}2 = \int_{\omega^{N \setminus \{i\}} \in \Omega^{N \setminus \{i\}}} \int_{x_{i} \in \Xx_{i}} \Ii^{O}(\omega^{N \setminus \{i\}} \times \{x_{i}^{O_{\mathrm{min}}}\}) \; \; \mu_{\omega^{N \setminus \{i\}}} (dx_{i}) \; \lambda (d\omega^{N \setminus \{i\}}).
\]

The term $\Ii^{O}(\omega^{N \setminus \{i\}} \times \{x_{i}^{O_{\mathrm{min}}}\})$ is constant with respect to $x_{i}$, therefore it can be brought outside of the inner integral to give,

\begin{eqnarray}
\mathrm{I}2 & = & \int_{\omega^{N \setminus \{i\}} \in \Omega^{N \setminus \{i\}}} \Ii^{O}(\omega^{N \setminus \{i\}} \times \{x_{i}^{O_{\mathrm{min}}}\}) \int_{x_{i} \in \Xx_{i}} \; \; \mu_{\omega^{N \setminus \{i\}}} (dx_{i}) \; \lambda (d\omega^{N \setminus \{i\}}). \nonumber \\
\mathrm{I}2 & = & \int_{\omega^{N \setminus \{i\}} \in \Omega^{N \setminus \{i\}}} \Ii^{O}(\omega^{N \setminus \{i\}} \times \{x_{i}^{O_{\mathrm{min}}}\})  \; \lambda (d\omega^{N \setminus \{i\}}). \nonumber 
\end{eqnarray}

The construction of the product space $(\Omega,\mathcal{F},\Pp)$ ensures, 

\[
 \lambda (d\omega^{N \setminus \{i\}}) = \lambda_{x_{i}^{O_{\mathrm{min}}}} (d\omega^{N \setminus \{i\}}).
\]

Therefore $\mathrm{I}2$ can be expressed as,

\[
\mathrm{I}2 = \! \! \!  \int_{\omega^{N \setminus \{i\}} \in \Omega^{N \setminus \{i\}}} \Ii^{O}(\omega^{N \setminus \{i\}} \times \{x_{i}^{O_{\mathrm{min}}}\})  \; \lambda_{x_{i}^{O_{\mathrm{min}}}} (d\omega^{N \setminus \{i\}}).
\]

$\mathrm{I}2$ is the expectation of the random variable $\Ee \Ii^{O}$, given player $i$ has expressed $\{x_{i}^{O_{\mathrm{min}}}\}$.  Hence,

\begin{equation}
\mathrm{I}2  =  \Pr(O | \{x_{i}^{O_{\mathrm{min}}}\}). \label{eq:dcDI2cv1}
\end{equation}

Substituting Equations (\ref{eq:dcDI2bv1}) and (\ref{eq:dcDI2cv1}) in Equation (\ref{eq:dcDI2v1}) yields the result.  
\end{proof}

\begin{lemma}
For a $GVG(\Omega,\mathcal{F},\Pp, \Ww)$ and a player $i$,

\[
\Pr(O \_ IC_{i}^{\delta})  =  \Pr(O | \{x_{i}^{O_{\mathrm{max}}}\}) - \Pr(O).
\]

\label{lem:pricDprobv1}
\end{lemma}

\begin{proof}
Let $\omega^{N \setminus \{i\}} = \omega^{N} \setminus x_{i}$, then by Definitions \ref{def:incc}, \ref{def:critD} and \ref{def:imaxmin}, the indicator function $\Ii^{O \_ IC_{i}^{\delta}} : \Omega^{N}  \rightarrow \{0,1\}$ for the set $O \_ IC_{i}^{\delta}$ is given by,

\[ \Ii^{O \_ IC_{i}^{\delta}}(\omega^{N}) = \left\{ \begin{array}{ll}
						1 & \mbox{if $\Ww(\omega^{N \setminus \{i\}} \times \{x_{i}^{O_{\mathrm{max}}} \}) = O$ and $\Ww(\omega^{N}) \neq O$}; \\
						0 & \mbox{otherwise}.
					\end{array}
			\right. \]

$\Ii^{O \_ IC_{i}^{\delta}}$ can be expressed as follows, where $\Ii^{O}$ is the indicator function for the event being classified as outcome $O$,

\begin{equation}
\Ii^{O \_ IC_{i}^{\delta}} = \Ii^{O}(\omega^{N \setminus \{i\}} \times \{x_{i}^{O_{\mathrm{max}}} \} ) - \Ii^{O}( \omega^{N}).
\label{eq:indiciD}
\end{equation}

Integrating $\Ii^{O \_ IC_{i}^{\delta}}$ over all the events in the $GVG$ creates the expectation of the random variable $\Ee \Ii^{O \_ IC_{i}^{\delta}}$, which we can interpret as the probability of an event $\omega^{N}$ being in the set $O \_ IC_{i}^{\delta}$.

\begin{eqnarray}
\Pr(O \_ IC_{i}^{\delta}) & = & \int_{\omega^{N} \in \Omega^{N}} \Ii^{O}(\omega^{N \setminus \{i\}} \times \{x_{i}^{O_{\mathrm{max}}}\}) \;  - \; \Ii^{O}(\omega^{N}) \; \; \Pp (d\omega^{N}).\nonumber \\
\Pr(O \_ IC_{i}^{\delta}) & = & \int_{\omega^{N} \in \Omega^{N}} \Ii^{O}(\omega^{N \setminus \{i\}} \times \{x_{i}^{O_{\mathrm{max}}}\}) \; \; \Pp (d\omega^{N}) \; \nonumber \\
 && \hspace{.75cm} - \; \int_{\omega^{N} \in \Omega^{N}} \Ii^{O}(\omega^{N}) \; \; \Pp (d\omega^{N}).\nonumber \\
\Pr(O \_ IC_{i}^{\delta}) & = & \mathrm{I}1 - \mathrm{I}2. \label{eq:icDI2v1}
\end{eqnarray}

Here, $\mathrm{I}2$ is the expectation of the random variable $\Ee \Ii^{O}$, which we can interpret as the probability of an event $\omega^{N}$ being in the set $O$.

\begin{equation}
\mathrm{I}2 = \Pr(O). \label{eq:icDI2bv1}
\end{equation}

Using Lemma \ref{lem:multiint-i}, $\mathrm{I}1$ can be expressed as follows,

\[
\mathrm{I}1 = \int_{\omega^{N \setminus \{i\}} \in \Omega^{N \setminus \{i\}}} \int_{x_{i} \in \Xx_{i}} \Ii^{O}(\omega^{N \setminus \{i\}} \times \{x_{i}^{O_{\mathrm{max}}}\}) \; \; \mu_{\omega^{N \setminus \{i\}}} (dx_{i}) \; \lambda (d\omega^{N \setminus \{i\}}).
\]

The term $\Ii^{O}(\omega^{N \setminus \{i\}} \times \{x_{i}^{O_{\mathrm{max}}}\})$ is constant with respect to $x_{i}$, therefore it can be brought outside of the inner integral to give,

\begin{eqnarray}
\mathrm{I}1 & = & \int_{\omega^{N \setminus \{i\}} \in \Omega^{N \setminus \{i\}}} \Ii^{O}(\omega^{N \setminus \{i\}} \times \{x_{i}^{O_{\mathrm{max}}}\}) \int_{x_{i} \in \Xx_{i}} \; \; \mu_{\omega^{N \setminus \{i\}}} (dx_{i}) \; \lambda (d\omega^{N \setminus \{i\}}). \nonumber \\
\mathrm{I}1 & = & \int_{\omega^{N \setminus \{i\}} \in \Omega^{N \setminus \{i\}}} \Ii^{O}(\omega^{N \setminus \{i\}} \times \{x_{i}^{O_{\mathrm{max}}}\})  \; \lambda (d\omega^{N \setminus \{i\}}). \nonumber 
\end{eqnarray}

The construction of the product space $(\Omega,\mathcal{F},\Pp)$ ensures,

\[
\lambda (d\omega^{N \setminus \{i\}}) = \lambda_{x_{i}^{O_{\mathrm{max}}}} (d\omega^{N \setminus \{i\}}).
\]

Therefore $\mathrm{I}1$ can be expressed as,

\[
\mathrm{I}1 = \int_{\omega^{N \setminus \{i\}} \in \Omega^{N \setminus \{i\}}}  \Ii^{O}(\omega^{N \setminus \{i\}} \times \{x_{i}^{O_{\mathrm{max}}}\})  \;  \lambda_{x_{i}^{O_{\mathrm{max}}}} (d\omega^{N \setminus \{i\}}). 
\]

$\mathrm{I}1$ is the expectation of the random variable $\Ee \Ii^{O}$, given player $i$ has expressed $\{x_{i}^{O_{\mathrm{max}}}\}$.  Hence,

\begin{equation}
\mathrm{I}1 = \Pr(O | \{x_{i}^{O_{\mathrm{max}}}\}).
\label{eq:icDI2cv1}
\end{equation}

Substituting Equations (\ref{eq:icDI2bv1}) and (\ref{eq:icDI2cv1}) in Equation (\ref{eq:icDI2v1}) yields the result.  
\end{proof}

\begin{corol}
For a $GVG(\Omega,\mathcal{F},\Pp, \Ww)$ and a player $i$,

\[
\Pr(O \_ TC_{i}^{\delta}) = \Pr(O | \{x_{i}^{O_{\mathrm{max}}}\})- \Pr(O | \{x_{i}^{O_{\mathrm{min}}}\}).
\]

\label{cor:prtcDprobv1}
\end{corol}

\begin{proof}
By Lemmas \ref{lem:prdcDprobv1} and \ref{lem:pricDprobv1},

\begin{eqnarray}
\Pr(O \_ DC_{i}^{\delta})  & = & \Pr(O) - \Pr(O | \{x_{i}^{O_{\mathrm{min}}}\}). \label{eq:tcDeq1v1} \\
\Pr(O \_ IC_{i}^{\delta})  & = & \Pr(O | \{x_{i}^{O_{\mathrm{max}}}\}) - \Pr(O). \label{eq:tcDeq2v1}
\end{eqnarray}

By Definition \ref{def:totc},

\begin{equation}
\Pr(O \_ TC_{i}^{\delta}) = \Pr(O \_ DC_{i}^{\delta}) + \Pr(O \_ IC_{i}^{\delta}). \label{eq:tcDeq3v1}
\end{equation}

Substituting Equations (\ref{eq:tcDeq1v1}) and (\ref{eq:tcDeq2v1}) in Equation (\ref{eq:tcDeq3v1}) yields the result.  
\end{proof}

\subsection{Criticality 0}

This section provides the results for the more specialised criticality 0.  


\begin{lemma}
For a $GVG(\Omega,\mathcal{F},\Pp, \Ww)$ and a player $i$,

\[
\Pr(O \_ DC_{i}^{0})  =  \Pr(O) - \Pr(O | \{x_{i}^{O_{\mathrm{min}}}\}).
\]

\label{lem:prdc0probv1}
\end{lemma}

\begin{proof}
Let $\omega^{N \setminus \{i\}} = \omega^{N} \setminus x_{i}$, then by Definitions \ref{def:decc}, \ref{def:crit0} and \ref{def:imaxmin}, the indicator function $\Ii^{O \_ DC_{i}^{0}} : \Omega^{N}  \rightarrow \{0,1\}$ for the set $O \_ DC_{i}^{0}$ is given by,

\[ \Ii^{O \_ DC_{i}^{0}}(\omega^{N}) = \left\{ \begin{array}{ll}
						1 & \mbox{if $\Ww(\omega^{N}) = O$ and $\Ww(\omega^{N \setminus \{i\}} \times \{x_{i}^{O_{\mathrm{min}}} \}) \neq O$}; \\
						0 & \mbox{otherwise}.
					\end{array}
			\right. \]

$\Ii^{O \_ DC_{i}^{0}}$ can be expressed as follows, where $\Ii^{O}$ is the indicator function for the event being classified as outcome $O$,

\begin{equation}
\Ii^{O \_ DC_{i}^{0}} = \Ii^{O}( \omega^{N}) - \Ii^{O}(\omega^{N \setminus \{i\}} \times \{x_{i}^{O_{\mathrm{min}}} \} ).
\label{eq:inddci0}
\end{equation}

Comparing Equations (\ref{eq:inddci0}) and (\ref{eq:inddciD}) gives,
\[ \Ii^{O \_ DC_{i}^{0}} = \Ii^{O \_ DC_{i}^{\delta}}.\] Therefore, 
\begin{equation}
\Pr(O \_ DC_{i}^{0})  = \Pr(O \_ DC_{i}^{\delta}).
\end{equation} 
And hence, an application of Lemma \ref{lem:prdcDprobv1} yields the result.  
\end{proof}

\begin{lemma}
For a $GVG(\Omega,\mathcal{F},\Pp, \Ww)$ and a player $i$,

\[
\Pr(O \_ IC_{i}^{0} | \{x_{i}^{O_{\mathrm{min}}}\})  =  \Pr(O| \{x_{i}^{O_{\mathrm{max}}}\}) - \Pr(O | \{x_{i}^{O_{\mathrm{min}}}\}).
\]

\label{lem:pric0probv1}
\end{lemma}

\begin{proof}
Let $\omega^{N \setminus \{i\}} = \omega^{N} \setminus x_{i}$, then by Definitions \ref{def:incc}, \ref{def:crit0} and \ref{def:imaxmin}, the indicator function $\Ii^{O \_ IC_{i}^{0}} : \Omega^{N} \rightarrow \{0,1\}$ for the set $O \_ IC_{i}^{0}$ is given by,

\[ \Ii^{O \_ IC_{i}^{0}}(\omega^{N}) = \left\{ \begin{array}{ll}
						1 & \mbox{if $\{x_{i}\} = \{x_{i}^{O_{\mathrm{min}}}\}$, $\Ww(\omega^{N}) \neq O$,}
 \\ & \mbox{and $\Ww(\omega^{N \setminus \{i\}} \times \{x_{i}^{O_{\mathrm{max}}} \}) = O$}; \\
						0 & \mbox{otherwise}.
					\end{array}
			\right. \]

Let $\Ii^{\{ x_{i}^{O_{\mathrm{min}}} \} }$ be the indicator function for $\{x_{i}\} = \{x_{i}^{O_{\mathrm{min}}}\}$, and let $\Ii^{O}$ be the indicator function for an event  $\omega^{N}$ being classified as outcome $O$, then the indicator function $\Ii^{O \_ IC_{i}^{0}}$ can be expressed as, 

\begin{equation}
\Ii^{O \_ IC_{i}^{0}}(\omega^{N}) = \Ii^{ \{ x_{i}^{O_{\mathrm{min}}} \} } \left( \Ii^{O}( \omega^{N \setminus \{i\}} \times \{x_{i}^{O_{\mathrm{max}}} \}) - \Ii^{O}(\omega^{N \setminus \{i\}} \times \{x_{i}^{O_{\mathrm{min}}}\}) \right).
\label{eq:ic0pre1}
\end{equation}

Integrating Equation (\ref{eq:ic0pre1}) with respect to the sigma finite marginal measure $\lambda_{x_{i}^{O_{\mathrm{min}}}}$ produces the conditional expectation of the random variable $\Ee \Ii^{O \_ IC_{i}^{0}}$ given player $i$ has expressed $\{x_{i}^{O_{\mathrm{min}}}\}$, which we interpret as $\Pr(O \_ IC_{i}^{0} | \{x_{i}^{O_{\mathrm{min}}}\})$.  As an added bonus, it also allows the removal of the $\Ii^{ \{ x_{i}^{O_{\mathrm{min}}} \}}$ indicator function.  

\begin{equation}
\Pr(O \_ IC_{i}^{0} | \{x_{i}^{O_{\mathrm{min}}}\}) = \! \! \! \int_{\omega^{N \setminus \{i\}} \in \Omega^{N \setminus \{i\}}} \! \! \! \! \! \! \! \! \! \! \! \! \! \! \! \! \! \! \! \! \! \! \! \! \! \! \! \! \Ii^{O}( \omega^{N \setminus \{i\}} \times \{x_{i}^{O_{\mathrm{max}}} \}) - \Ii^{O}(\omega^{N \setminus \{i\}} \times \{x_{i}^{O_{\mathrm{min}}}\}) \lambda_{x_{i}^{O_{\mathrm{min}}}} (d\omega^{N \setminus \{i\}}).
\label{eq:ic0pre2}
\end{equation}

The integral in Equation (\ref{eq:ic0pre2}) can be split to give,

\begin{eqnarray}
\Pr(O \_ IC_{i}^{0} | \{x_{i}^{O_{\mathrm{min}}}\}) & = &    \int_{\omega^{N \setminus \{i\}} \in \Omega^{N \setminus \{i\}}}  \Ii^{O}( \omega^{N \setminus \{i\}} \times \{x_{i}^{O_{\mathrm{max}}} \}) \; \; \lambda_{x_{i}^{O_{\mathrm{min}}}} (d\omega^{N \setminus \{i\}}), \nonumber \\
 & - & \int_{\omega^{N \setminus \{i\}} \in \Omega^{N \setminus \{i\}}} \Ii^{O}(\omega^{N \setminus \{i\}}  \times \{x_{i}^{O_{\mathrm{min}}}\}) \; \; \lambda_{x_{i}^{O_{\mathrm{min}}}} (d\omega^{N \setminus \{i\}}). \nonumber 
\end{eqnarray}
Hence,
\begin{equation}
\Pr(O \_ IC_{i}^{0} | \{x_{i}^{O_{\mathrm{min}}}\})  =  \mathrm{I}1 - \mathrm{I}2. \label{eq:ic0I2v2}
\end{equation}

$\mathrm{I}2$ is the expectation of the random variable $\Ee \Ii^{O}$, given player $i$ has expressed $\{x_{i}^{O_{\mathrm{min}}}\}$.  Thus,

\begin{equation}
\mathrm{I}2 = \Pr(O | \{x_{i}^{O_{\mathrm{min}}}\}).
\label{eq:ic0I2bv2}
\end{equation}

The construction of the product space $(\Omega,\mathcal{F},\Pp)$ ensures,

\[
\lambda_{x_{i}^{O_{\mathrm{min}}}} (d\omega^{N \setminus \{i\}}) = \lambda_{x_{i}^{O_{\mathrm{max}}}} (d\omega^{N \setminus \{i\}}).
\]

Therefore, $\mathrm{I}1$ can be expressed as,

\[
\mathrm{I}1 =  \int_{\omega^{N \setminus \{i\}} \in \Omega^{N \setminus \{i\}}} \Ii^{O}(\omega^{N \setminus \{i\}} \times \{x_{i}^{O_{\mathrm{max}}}\}) \; \; \lambda_{x_{i}^{O_{\mathrm{max}}}} (d\omega^{N}). 
\]

$\mathrm{I}1$ is the expectation of the random variable $\Ee \Ii^{O}$, given player $i$ has expressed $\{x_{i}^{O_{\mathrm{max}}}\}$.  Hence,

\begin{equation}
\mathrm{I}1 = \Pr(O | \{x_{i}^{O_{\mathrm{max}}}\}).
\label{eq:ic0I2cv2}
\end{equation}

Substituting Equations (\ref{eq:ic0I2bv2}) and (\ref{eq:ic0I2cv2}) in Equation (\ref{eq:ic0I2v2}) yields the result.  
\end{proof}

\begin{lemma}
For a $GVG(\Omega,\mathcal{F},\Pp, \Ww)$ and a player $i$,

\[
\Pr(O \_ IC_{i}^{0})  = \Pr(\{x_{i}^{O_{\mathrm{min}}}\}) \times \left( \Pr(O| \{x_{i}^{O_{\mathrm{max}}}\}) - \Pr(O | \{x_{i}^{O_{\mathrm{min}}}\}) \right).
\]

\label{lem:pric0prob}
\end{lemma}

\begin{proof}
Taking Lemma \ref{lem:pric0probv1} and multiplying by $\Pr(\{x_{i}^{O_{\mathrm{min}}}\})$ gives,

\begin{eqnarray*}
\Pr(\{x_{i}^{O_{\mathrm{min}}}\}) \times \Pr(O \_ IC_{i}^{0}) & = & \Pr(\{x_{i}^{O_{\mathrm{min}}}\}) \times \left( \Pr(O| \{x_{i}^{O_{\mathrm{max}}}\}) - \Pr(O | \{x_{i}^{O_{\mathrm{min}}}\}) \right).\\
\Pr(\{x_{i}^{O_{\mathrm{min}}}\} \cap O \_ IC_{i}^{0}) & = & \Pr(\{x_{i}^{O_{\mathrm{min}}}\}) \times \left( \Pr(O| \{x_{i}^{O_{\mathrm{max}}}\}) - \Pr(O | \{x_{i}^{O_{\mathrm{min}}}\}) \right).
\end{eqnarray*}

By Definitions \ref{def:incc}, \ref{def:crit0}, and \ref{def:imaxmin} $(O \_ IC_{i}^{0}) = (\{x_{i}^{O_{\mathrm{min}}}\} \cap O \_ IC_{i}^{0})$, which yields the result.

\end{proof}

\begin{corol}
For a $GVG(\Omega,\mathcal{F},\Pp, \Ww)$ and a player $i$, 

\begin{eqnarray*}
\Pr(O \_ TC_{i}^{0}) & = & \Pr(O) - \Pr(O | \{x_{i}^{O_{\mathrm{min}}}\}) + \\
&  & \Pr(\{x_{i}^{O_{\mathrm{min}}}\}) \times \left( \Pr(O| \{x_{i}^{O_{\mathrm{max}}}\}) - \Pr(O | \{x_{i}^{O_{\mathrm{min}}}\}) \right).
\end{eqnarray*}

\label{cor:prtc0probv1}
\end{corol}

\begin{proof}
By Lemma \ref{lem:prdc0probv1},

\begin{equation}
\Pr(O \_ DC_{i}^{0})  = \Pr(O) - \Pr(O | \{x_{i}^{O_{\mathrm{min}}}\}). \label{eq:tc0eq1v1}
\end{equation}

By Lemma \ref{lem:pric0prob},

\begin{equation}
\Pr(O \_ IC_{i}^{0})  = \Pr(\{x_{i}^{O_{\mathrm{min}}}\}) \times \left( \Pr(O| \{x_{i}^{O_{\mathrm{max}}}\}) - \Pr(O | \{x_{i}^{O_{\mathrm{min}}}\}) \right).
\label{eq:tc0eq2v1}
\end{equation}

By Definition \ref{def:totc},

\begin{equation}
\Pr(O \_ TC_{i}^{0}) = \Pr(O \_ DC_{i}^{0}) + \Pr(O \_ IC_{i}^{0}). \label{eq:tc0eq3v1}
\end{equation}

Substituting Equations (\ref{eq:tc0eq1v1}) and (\ref{eq:tc0eq2v1}) in Equation (\ref{eq:tc0eq3v1}) yields the result.  
\end{proof}

\subsection{Summary}

Using measure theory, this section has shown that the criticality based voting power measures reduce to a simple expression, involving at most three, or four, simple probabilities: $\Pr(O)$, $\Pr(O| \{x_{i}^{O_{\mathrm{max}}}\})$, $\Pr(O | \{x_{i}^{O_{\mathrm{min}}}\})$, and $\Pr(\{x_{i}^{O_{\mathrm{min}}}\})$.  This result has been produced for the most general type of voting game possible, with absolutely no restriction on how the game is constructed.  The voting actions of a player can include a selection from an infinite set, or it can be as simple as voting ``yes'' or ``no''.  The decision rule of the game can range from a simple majority, to the most complex non-monotone rule you can imagine.

To the best of our knowledge, this represents the first time that voting power analysis has been given such a fundamental and mathematically justified basis.  We hope that these insights will inspire a renewed vigour in voting power analysis, along the lines of the renaissance in artificial neural networks, sparked by a similarly mathematical justification.

\section{The Standard Techniques Revisited}
\label{sect:vptr}

In Section \ref{sect:vpt} we started the process of understanding the standard techniques by showing how they can be expressed in terms of criticality functions.  In the previous section, we analysed these criticality functions in the case of our generalisation of a voting game.  In this section, we bring this to a logical conclusion by providing a description of these techniques in the same probabilistic terms.  

\begin{theory}
For a $GVG(\Omega,\mathcal{F},\Pp, \Ww)$ and a player $i$, the standard voting power techniques are calculating,

\begin{eqnarray*}
\mathrm{ShapleyShubik} & = & \Pr(\{x_{i}^{O_{\mathrm{min}}}\}) \times \left( \Pr(O| \{x_{i}^{O_{\mathrm{max}}}\}) - \Pr(O | \{x_{i}^{O_{\mathrm{min}}}\}) \right). \\
\mathrm{Banzhaf} & = & \Pr(O | \{x_{i}^{O_{\mathrm{max}}}\})- \Pr(O | \{x_{i}^{O_{\mathrm{min}}}\}). \\
\mathrm{Straffin} & = &  \Pr(O | \{x_{i}^{O_{\mathrm{max}}}\})- \Pr(O | \{x_{i}^{O_{\mathrm{min}}}\}). \\
\mathrm{Johnston} & = &  \Pr(O) - \Pr(O | \{x_{i}^{O_{\mathrm{min}}}\}). \\
\mathrm{Coleman \; Initiate \; Action} & = & \frac{\Pr(O | \{x_{i}^{O_{\mathrm{max}}}\}) - \Pr(O)}{1 - \Pr(O)}.\\
\mathrm{Coleman \; Prevent \; Action} & = &  \frac{\Pr(O) - \Pr(O | \{x_{i}^{O_{\mathrm{min}}}\})}{\Pr(O)}.
\end{eqnarray*}
\label{theory:vottech}
\end{theory}

\begin{proof}
In Section \ref{sect:vpt} it was shown that the standard techniques are calculating the probability of a criticality set.  Taking the results from that section, and rephrasing them in the generalised context gives,

\begin{eqnarray}
\mathrm{ShapleyShubik} & = & \Pr(O \_ IC^{0}). \label{eq:ssvpt} \\
\mathrm{Banzhaf} & = & \Pr(O \_ TC^{\delta}). \label{eq:bzvpt} \\
\mathrm{Straffin} & = &  \Pr(O \_ TC^{\delta}). \label{eq:stvpt} \\
\mathrm{Johnston} & = &  \Pr(O \_ DC^{0}). \label{eq:jvpt} \\
\mathrm{Coleman \; Initiate \; Action} & = & \Pr(O \_ IC^{\delta} \; | \; \overline{O} \;).  \label{eq:cia1vpt} \\
\mathrm{Coleman \; Prevent \; Action} & = &  \Pr(O \_ DC^{\delta} \; | \; O \;). \label{eq:cpa1vpt}
\end{eqnarray}

Let's examine the Coleman measures, and make them a little easier to handle.  By the standard rules of probability, 

\begin{eqnarray}
\mathrm{Coleman \; Initiate \; Action} & = & \Pr(O \_ IC^{\delta} \; | \; \overline{O} \;) = \frac{\Pr(O \_ IC^{\delta} \cap \overline{O})}{\Pr(\overline{O})}. \label{eq:cia2vpt} \\
\mathrm{Coleman \; Prevent \; Action} & = &  \Pr(O \_ DC^{\delta} \; | \; O \;) = \frac{\Pr(O \_ DC^{\delta} \cap O)}{\Pr(O)}. \label{eq:cpa2vpt}
\end{eqnarray}

By Definitions \ref{def:incc} and \ref{def:decc} 
\[(O \_ DC^{\delta}) = (O \_ DC^{\delta} \cap O),\]
and 
\[(O \_ IC^{\delta}) = (O \_ IC^{\delta} \cap  \overline{O}).\]
Replacing these terms in Equations (\ref{eq:cia2vpt}) and (\ref{eq:cpa2vpt}), and using \mbox{$\Pr(\overline{O}) = 1 - \Pr(O)$} gives,

\begin{eqnarray}
\mathrm{Coleman \; Initiate \; Action} & = & \frac{\Pr(O \_ IC^{\delta} \cap \overline{O})}{\Pr(\overline{O})} \label{eq:cia3vpt} \\ 
&= & \frac{\Pr(O \_ IC^{\delta})} {\Pr(\overline{O})} = \frac{\Pr(O \_ IC^{\delta})} {1 - \Pr(O)}. \nonumber\\
\mathrm{Coleman \; Prevent \; Action} & = &  \frac{\Pr(O \_ DC^{\delta} \cap O)}{\Pr(O)} = \frac{\Pr(O \_ DC^{\delta})} {\Pr(O)}. \label{eq:cpa3vpt}
\end{eqnarray}

The proof is completed by using Corollary \ref{cor:prtcDprobv1} and Lemmas \ref{lem:prdcDprobv1}, \ref{lem:pricDprobv1}, \ref{lem:prdc0probv1}, and \ref{lem:pric0prob},  to replace terms in Equations (\ref{eq:ssvpt}), (\ref{eq:bzvpt}), (\ref{eq:stvpt}), (\ref{eq:jvpt}), (\ref{eq:cia3vpt}), and (\ref{eq:cpa3vpt}).  
\end{proof}

Let's examine this result in more familiar terms.  The standard techniques are most commonly used in voting games which can result in either a winning or losing outcome.  The previous theorem tells us that the standard techniques are in fact calculating the following,

\begin{eqnarray*}
\mathrm{ShapleyShubik} & = & \Pr(\mbox{Player } i \mbox{ votes no}) \\
&& \times \space \left( \Pr(\mbox{Winning}| \mbox{Player } i \mbox{ votes yes}) - \Pr(\mbox{Winning}| \mbox{Player } i \mbox{ votes no})  \right). \\
\mathrm{Banzhaf} & = & \Pr(\mbox{Winning}| \mbox{Player } i \mbox{ votes yes}) - \Pr(\mbox{Winning}| \mbox{Player } i \mbox{ votes no}) . \\
\mathrm{Straffin} & = &  \Pr(\mbox{Winning}| \mbox{Player } i \mbox{ votes yes}) - \Pr(\mbox{Winning}| \mbox{Player } i \mbox{ votes no}) . \\
\mathrm{Johnston} & = &  \Pr(\mbox{Winning}) - \Pr(\mbox{Winning}| \mbox{Player } i \mbox{ votes no}) .
\end{eqnarray*}
and 
\begin{eqnarray*}
\mathrm{Coleman \; Initiate \; Action} & = & \frac{\Pr(\mbox{Winning} |  \mbox{Player } i \mbox{ votes yes}) - \Pr(\mbox{Winning})}{1-  \Pr(\mbox{Winning})}.\\
\mathrm{Coleman \; Prevent \; Action} & = &  \frac{\Pr(\mbox{Winning}) - \Pr(\mbox{Winning} | \mbox{Player } i \mbox{ votes no})}{\Pr(\mbox{Winning})}.
\end{eqnarray*}

As an unexpected bonus of our general framework the above results apply irrespective of the number of ways in which the player can vote.  Whether it be a simple ``yes/no'' game, a game with abstentions, or a choice from a continuum of approval values.  Moreover, all these techniques are expressed in terms of three (or four) simple probabilities.   This makes them easy to comprehend and transparent.  But perhaps, the most important consequence, is being able to express these techniques independent of a probability model.  A huge problem with the standard techniques is their requirement for a specific probability model, making it impractical to use them in a game with a different probability model.  A drawback not faced by our measure theoretic reinterpretation.

There is one final point to note. The results of this theorem contradict the work of Straffin.  He long argued that the Shapley-Shubik index and the Banzhaf measure were calculating the same thing, albeit with different probability models.  He based his ideas upon the fact that his Homogeneity Assumption measure was numerically equivalent to the Shapley-Shubik index.  Which is true, but only for simple ``yes/no'' games without abstentions.  A more complex game, even something as simple as allowing a few players to abstain, will break this numerical equivalency, and show that the Homogeneity Assumption measure is not the same as the Shapley-Shubik index. We can see this quite easily with an example.  Imagine a game where the number of possible player actions is large, perhaps even infinite.  In such a game, it is reasonable to assume that $\Pr(\{x_{i}^{O_{\mathrm{min}}}\}) \rightarrow 0$, making the Shapley-Shubik index tend to zero as well.  But, both the Banzhaf measure, and the Straffin indices, will not tend to zero (unless, of course, the player has no influence on the outcome).  Therefore, the Shapley-Shubik index is inequivalent to the Banzhaf measure and the Straffin indices.

\section{Discussion}

We have produced a generalised description of criticality using measure theory.  This work has brought together all the known voting power techniques under one unifying framework.  Our mathematical framework allows voting power to be calculated in both simple and complex games, from basic ``yes/no'' voting, to voting with abstentions, and even abstract non-monotone voting games with multiple outcomes.  Furthermore, our measures are constructed independent of the underlying probability distribution of the players.  This allows the use of the correct probability model for the game in question, instead of the model imposed arbitrarily by a voting power technique.  

A huge obstacle currently faced by voting power theorists is persuading other researchers, politicians, and the general public, how their chosen standard technique measures voting power better than any other technique. However, in this work, we have been able to show mathematically that these techniques are calculating just three, or four, standard probabilities.  Thus, any debate about the superiority of one index over another is redundant.  In lieu of analysing voting power using the standard techniques, we advocate the adoption of these probabilities.  Probabilities are intuitive notions of influence, and are widely used in every day life. Expressing voter influence in terms of probability should lead to greater acceptance within society.  
 
The intuitive nature of these results have an additional benefit. It makes it easier to explain these ideas to a wider audience.  The difficult task of trying to explain to a politician what a Banzhaf measure is, or what Total Criticality $\delta$ means, can be replaced with the easier to understand common sense approach to measuring total influence.  For instance, a hypothetical conversation with a politician might go something like this:

``Imagine, that you vote against a particular motion, then there is a 30\% chance of it passing.  However, should you choose vote in favour, the chance of it passing goes up to 44\%.  Therefore, your total voting power in this game is 44\% - 30\% = 14\%.''  

Explained this way, voting power is both obvious and easy for a layperson to understand.  And it should make it easier for institutions to adopt voting procedures that respect the notions of a fair distribution of voting power.  With more and more people better able to understand voting power analysis, it will become easier to create democratically fairer institutions, and ultimately improve democracy itself.

\section{Conclusion}

Everyone is subject to the decisions made, or not made, by voting games.  Whether it is a decision to act on climate change by the United Nations, or a decision to collect your waste fortnightly by the local council.  Arguably, they are one of the most influential types of game studied by game theorists.

The need to create players with unequal power in large democratic institutions is well established (for example, the E.U. Council of Ministers has a voting game where the players represent populations of unequal size).  However, up until now, there has not been a widely accepted, mathematically justifiable, method for measuring the power of a player in a voting game - a situation which has impeded this research field to have its proposals adopted wholeheartedly by large democratic institutions.  

We have shown that all voting power measures based on the concept of criticality can be reduced to a simple set of probabilistic expressions.    Analysing voting games using just these simple probabilities brings important advantages over the status quo of the standard techniques.  First, it allows the results to be presented to, and understood by, a much wider range of people.  Second, it ends the debate over the superiority of one technique over another.  And third, these probabilities can be calculated for any voting game, irrespective of probability model (unlike the standard techniques).

Crucially, our results apply to all possible voting games; from the simple ``yes/no'' voting games to the abstract ones with non-monotone decision rules and multiple candidates.  As such, we feel that our paper provides the necessary mathematical tools to help build better democracies.

%


%

\appendix

\section{The Deegan-Packel, and Holler Public Good Indices}
\label{appendix:dphp}

All the results given in this paper so far have been applicable to any type of game.  Unfortunately, the Deegan-Packel, and the Holler Public Good indices rely upon a concept called a minimum winning coalition (a minimum winning event).  These indices can only be applied to games with monotonic decision rules.  As this restricts the type of games for which they can be applied, a discussion of these indices has been put off up until now.    

While the upcoming results can be understood without an understanding of monotonic decision rules, for completeness, they will be briefly explained.  

\subsection{Monotonic Decision Rules}

A monotonic decision rule induces an order upon the elements of $\Xx_{i}$, such that,

\[
\mathrm{if \;} \Ww(w^{N \setminus \{i\}} \times x_{i}) = O, \mathrm{\; then \;} \Ww(w^{N \setminus \{i\}} \times x'_{i}) = O, \mathrm{\; for \; all \;} x'_{i} \stackrel{O}{\geq} x_{i}.
\]

This order allows us to define a minimum winning event as an event which is classified as outcome $O$, but is no longer classified as $O$ when any of the players replace their current action with an action immediately below, with respect to the order induced upon them.  

If you think about this is terms of a simple ``yes/no'' game, a minimum winning coalition is a winning coalition in which every redundant player is removed; to leave only those necessary for the coalition to remain winning.

\subsection{The Indices}

The~\cite{DeeganPackel1978}, and~\cite{Holler1982} indices are incredibly similar.  Both indices can be described by the following.

\begin{enumerate}
\item Examine every minimum winning event.
\item Identify if it is decreasingly critical delta for player $i$.
\item If so, add $1$ (for Holler), or a fraction of $1$ (for Deegan-Packel) to a running count for player $i$.
\item Repeat until all minimum winning events have been examined.
\end{enumerate} 

In the Deegan-Packel index, the fraction that is added is a function of $\omega^{N}$ only, hence it can be absorbed within the $\Pr(d\omega)$ function of the game.  In other words, both the Deegan-Packel and Holler Public Good indices are the same, albeit with slightly different probability models. The actual fraction that is added in the Deegan-Packel index is inversely proportional to the number of players that express non-zero support in $\omega^{N}$.  Hence, the probability model of the Deegan-Packel index implies that events with more players expressing non-zero support are less likely to occur.  

Focusing upon the Deegan-Packel index, we note, from their paper, that their probability model assumes that only minimum winning events ($MWE$) will form.  Ergo, $\Pr(d\omega) = 0$ unless $\omega \in MWE$.  This allows the required integration to be carried out over all possible $\omega$ (as the probability model will ensure that only $MWE$s occur).

\[
\mathrm{DeeganPackel} = \int_{\omega \in \Omega } \; \; \Ii^{O \_ DC^{\delta}}(\omega) \; \;  \Pr(d\omega) \; = \; \Pr(O \_ DC^{\delta}).
\]

But, from Lemma \ref{lem:prdcDprobv1} we already know that

\[
\Pr(O \_ DC^{\delta}) = \Pr(O) - \Pr(O | \{x_{i}^{O_{\mathrm{min}}}\}).
\]

Hence, these indices can be expressed with the same probabilities we previously suggested.

\bibliographystyle{plain}

\bibliography{cond_prob_critique}

\end{document}